\documentclass[12pt]{amsart}
\usepackage{amssymb}
\usepackage{multicol}
\usepackage{colordvi}
\usepackage{graphicx}
\usepackage{pgf,epsfig, epsf}
\numberwithin{equation}{section}
\theoremstyle{plain}
\newtheorem{Th}{Theorem}[section]
\newtheorem{Cor}[Th]{Corollary}

\newtheorem{Lemma}[Th]{Lemma}
\numberwithin{equation}{section} \theoremstyle{definition}

\newtheorem{Ex}{Example}

\newtheorem{Def}[Th]{Definition}

\copyrightinfo{}{}
\newcounter{FNC}[page]
\def\fauxfootnote#1{{\addtocounter{FNC}{2}\Magenta{$^\fnsymbol{FNC}$}%
     \let\thefootnote\relax\footnotetext{\Magenta{$^\fnsymbol{FNC}$#1}}}}
\newcommand{\Z}{\mathbb Z}
\newcommand{\R}{\mathbb R}

\newcommand{\conv}{\operatorname{conv}}

\newcommand{\w}{\operatorname{w}}
\newcommand{\nls}{\operatorname{nls_\Delta}}
\newcommand{\ls}{\operatorname{ls_\Delta}}
\newcommand{\lss}{\operatorname{ls_\square}}

\title{Lattice Size of Plane Convex Bodies}

\author{Anthony Harrison}
\address{Department of Mathematics\\
        Kent State University\\
        Summit Street, Kent, OH 44242, USA}
\email{aharri60@kent.edu}
\thanks{Work of Soprunova and Tierney was partially supported by NSF Grant DMS-1156798}
\author{Jenya Soprunova}
\address{Department of Mathematics\\
        Kent State University\\
        Summit Street, Kent, OH 44242, USA}
\email{soprunova@math.kent.edu}
\urladdr{http://www.math.kent.edu/~soprunova/}
\author{Patrick Tierney}
\address{Department of Mathematics\\
University of Washington\\
Box 354350, Seattle, WA 98195,
USA}
\email{ptierney@uw.edu}

\begin{document}

\begin{abstract}
The lattice size  $\ls(P)$ of a lattice polygon $P$ with respect to the standard simplex $\Delta$ was introduced and studied by Castryck and Cools in the context of simplification of the defining equation of an algebraic curve. 
Earlier, Schicho provided an ``onion skins'' algorithm for mapping a lattice polygon $P$ into a small integer multiple of the standard simplex, based on passing
successively to the convex hull of the interior lattice points of $P$. Castryck and Cools showed that this algorithm computes the lattice size of $P$.
 
In this paper we show that for a plane convex body $P$  a reduced basis of $\Z^2$ computes the lattice size.
This provides a lattice reduction algorithm for computing the lattice size, which works for any convex body $P\subset\R^2$ and outperforms the ``onion skins'' algorithm in the case when $P$ is a lattice polygon.
\end{abstract}

\maketitle

%%%%%%%%%%%%%%%%%%%%%%%%%%%%%%%%%%%%%%%%%%%%%%%%%%%%%%%%%%%%%%%%%%%%%%%%%%%%%
\section*{Introduction}

The central object of this paper is the lattice size of $P$, where $P$ is a plane convex body. In the case when $P$ is a lattice polygon, this invariant 
was defined in \cite{CasCools} in the context of simplification of the defining equation of an algebraic curve.  This notion appeared implicitly earlier in the work of Arnold \cite{Arnold},  B\'ar\'any and Pach \cite{BarPach},  Brown and Kasprzyk \cite{BrownKasp}, and Lagarias and Ziegler \cite{ LagZieg}.

Recall that an {\it affine unimodular  transformation} $T\colon\R^d\to \R^d$ is a composition of the multiplication by a unimodular matrix $A\in {\rm GL}(d,\Z)$ and  a translation by an integer vector. We next reproduce the definition of the lattice size from~\cite{CasCools}, 
which now applies to a convex body $P\subset\R^d$.

\begin{Def} Let $X\subset\R^d$  be  a set with positive Jordan measure.
The {\it lattice size} $\operatorname{ls}_X(P)$  with respect to $X$  is the smallest  integer $l$ such that $T(P)$ is contained in the $l$-dilate  of $X$ for some affine unimodular transformation $T$. The corresponding matrix $A$ is then said to compute $\operatorname{ls}_X(P)$. 
\end{Def}

Note that when $X=[0,1]\times\R^{n-1}$,  the lattice size of $P$ with respect to $X$ is the lattice width $\w(P)$, an important invariant in lattice geometry and its applications.

In this paper we are studying $\ls(P)$, the lattice size  with respect to the standard simplex $\Delta\subset\R^d$. Let $f\in k[x^{\pm 1}_1,\dots, x^{\pm 1}_d]$ be a Laurent polynomial over a field $k$ and consider a hypersurface $H$ in $k^d$ defined by $f=0$.
The Newton polytope $P$ of $f$ is the convex hull of the exponent vectors that appear in $f$. Given a unimodular matrix $A=(a_{ij})$, a unimodular  change of variables $x_i=u_1^{a_{1i}}\dots u_d^{a_{di}}$ transforms $f$ into a Laurent polynomial
with the Newton polytope $A(P)$. Therefore $\ls(P)$ is the smallest possible total degree of the defining equation of $H$ under unimodular changes of variables.

Motivated by this observation, Schicho~\cite{Schicho}  provided an ``onion skins'' algorithm for mapping a lattice polygon $P$ into $l\Delta$ for a small integer $l$. In \cite{CasCools} Castryck and Cools proved that this algorithm computes $\ls(P)$. 
The algorithm is based on successively passing from an integer polygon to the convex hull of its interior lattice points. One can easily see that under this operation the lattice size drops by at least 3.
It was shown in  \cite{CasCools} that the lattice size $\ls(P)$ drops exactly by 3, unless $P$ belongs to a described list of exceptional cases.

Let $\square=[0,1]^d$ be the unit cube. In the case $d=2$ and a lattice polygon $P$, Castryck and Cools also provided in \cite{CasCools}  an ``onion skins'' algorithm for computing the lattice $\lss(P)$.
While the ``onion skins'' algorithm is very visual, it involves listing the interior lattice  points of $P$, which is time-consuming. It is also not feasible to generalize the ``onion skins'' algorithm to even dimension 3.   
 In~\cite{HarSopr} the lattice size $\lss(P)$ was recognized as a successive minimum of $K=\left(P+(-P)\right)^\circ$, the polar dual of the Minkowski sum of $P$ with $-P$, the reflection of $P$ in the origin.
 In dimension 2, the generalized Gauss Reduction algorithm of Kaib and Schnorr~\cite{KaibSch} computes the successive minima of an arbitrary origin-symmetric convex body $K$, and therefore it also computes
 the lattice size of the corresponding $P$, which can now be assumed to be an arbitrary convex body in $\R^2$. It is explained in~\cite{HarSopr} that in the case when $P$ is a lattice polygon, the algorithm  of Kaib and Schnorr outperforms the
  ``onion skins'' algorithm. The algorithm of Kaib and Schnorr was then extended in~\cite{HarSopr}  to dimension 3. Hence, in particular, the generalized basis reduction algorithm of~\cite{HarSopr} provides a fast way of computing $\lss(P)$ for 
 an arbitrary convex body $P\subset\R^3$.

Given a convex body $P\subset\R^2$, it was shown in~\cite{HarSopr} that if a basis $(h^1,h^2)$  of $\Z^2$ is reduced with respect to $K=\left(P+(-P)\right)^\circ$ (see Definition~\ref{D:def_red2}) then matrix $A$ with rows $h^1$  and $h^2$ 
computes $\lss(P)$. The main result of the current paper is formulated in Theorem~\ref{T:main} and Corollary~\ref{C:cormain}, where we prove a similar statement for $\ls(P)$. We define $\nls(P)$  to be the smallest $l\geq 0$ such that $P$ is contained in $l\Delta$  after a transformation which is  a composition of the multiplication by a unimodular  matrix of the form $\left[\begin{matrix}\pm 1&0\\0&\pm 1\end{matrix}\right]$ and  a lattice translation. We then show that if a lattice basis $(h^1,h^2)$ is reduced (with respect to $K$) then $\ls(P)=\nls(AP)$, where $A$ is the matrix whose rows are $h^1$ and $h^2$. Therefore, the generalized Gauss reduction algorithm of Kaib and Schnorr can be used to find the lattice size $\ls(P)$ of an arbitrary convex body $P\subset\R^2$. Not only this algorithm works for a much wider class of  $P\subset\R^2$ than the ``onion skins'' algorithm, it is also faster in the case when $P$ is a lattice polygon, as explained in~\cite{HarSopr}.    

It was further shown in~\cite{HarSopr} that if a lattice basis  $(h^1,h^2,h^3)$ of $\Z^3$ is reduced  then matrix $A$ with rows $h^1,h^2$, and $h^3$ computes  $\lss(P)$. In Example~\ref{E:counter_example} we demonstrate that this is not the case for $\ls(P)$ and therefore one cannot use basis reduction to compute $\ls(P)$ in  dimension 3.

%%%%%%%%%%%%%%%%%%%%%%%%%%%%%%%%%%%%%%%%%%%%%%%%%%%%%%%%%%%%%%%%%%%%%%%%%%%%%
%
\section{Definitions}

Recall that ${\rm GL}(d,\Z)$ is the set of square matrices of size $d$ that are invertible over $\Z$, that is, for $A\in {\rm GL}(d,\Z)$ we have $\det A=\pm 1$.  A map $T\colon \R^d\to \R^d$ of the form $T(x)=Ax+v$, where $x\in\R^d$, $A\in {\rm GL}(d,\Z)$, and  $v\in\Z^d$ is called an {\it affine unimodular map}. Unimodular maps preserve the integer lattice $\Z^d\subset\R^d$. 
For a lattice polytope  $P\subset \R^d$ we will simply write $AP$ for the image $T(P)$ of $P$ under the  linear transformation  $T\colon \R^d\to\R^d$ defined by $T(x)=Ax$.
A vector $v\in\Z^d$ is {\it primitive} if its components are relatively prime. 

%Two lattice polytopes are {\it unimodularly equivalent} if there exists a unimodular map  that maps one of them  to the other.

\begin{Def}\label{D:width} Let $P\subset\R^d$ be a convex body.
 The lattice width of  $P$ in the direction of $h\in\R^d$ is defined by
$$\w_h(P)=\max\limits_{x\in P} h\cdot x- \min_{x\in P} h\cdot x,
$$
where $ h\cdot x$ denotes the standard dot-product. The {\it lattice width} $\w(P)$  of $P$ is the minimum of $\w_h(P)$ over all non-zero primitive vectors $h\in \Z^d$. 
\end{Def}

Let $K:=(P+(-P))^{\circ}$ be the polar dual of the Minkowski sum of $P$ with $-P$, the reflection of $P$ in the origin.
Then $K$ is origin-symmetric and convex and it defines a norm on $\R^d$ by
$$\Vert h\Vert _K=\inf\{\lambda>0\mid h/\lambda \in K\}.$$
For details see, for example, \cite{Barvinok}. We then have
$$\Vert h\Vert_K =\inf\{\lambda>0\mid h\cdot x\leq\lambda {\rm\ for\ all\ }  x\in K^\circ\}=\max\limits_{x\in K^{\circ}} h\cdot x =\frac{1}{2}\w_h(K^{\circ})=\w_h(P).
$$
This, in particular, implies that the lattice  width of $P$ defines a convex function on $\R^d$. In what follows we will often write $\Vert h\Vert_K$ or simply $\Vert h\Vert$ for the lattice width $\w_h(P)$.

Let $0$ denote the origin in $\R^d$ and let $(e^1,\dots, e^d)$ be the standard basis of $\R^d$. Next, let 
$\Delta=\conv\{0,e^1,\dots, e^d\}\subset\R^d$ be  the standard simplex.  

\begin{Def}
The {\it lattice size} $\ls(P)$ of a convex body $P\subset\R^d$ with respect to the standard simplex $\Delta$ is the smallest $l\geq 0$ such that $P$ is contained in the $l$-dilate  $l\Delta$ of the standard simplex $\Delta$ after an affine unimodular transformation. A unimodular transformation $T$ and the corresponding matrix $A$  which minimize $l$ are said to compute $\ls(P)$.
\end{Def}

In~\cite{HarSopr} the lattice size $\lss(P)$  of a convex body $P\subset\R^d$ was recognized as a successive minimum of the corresponding $K$, which led to a fast lattice reduction algorithm for computing the lattice size with respect to the cube in dimensions 2 and 3.  In this paper we will show that given a reduced basis, one can easily recover $\ls(P)$.
Let $P\subset\R^2$ be a convex body  and let $K:=(P+(-P))^{\circ}$.

\begin{Def}\label{D:def_red2} We say that a basis $(h^1, h^2)$ of the integer lattice $\Z^2\subset\R^2$ is {\it reduced} (with respect to $K$) if 
$$\w_{h^i}(P)\leq \w_{h^1\pm h^2}(P) \ {\rm\ for\ }i=1,2.
$$
If we use the notation $\Vert h\Vert _K=\w_h(P)$, this condition rewrites as  
$$\Vert h^i\Vert_K \leq \Vert h^1\pm h^2\Vert_K\ {\rm\ for\ }i=1,2.$$
Usually one also requires that the basis vectors are ordered so that  $\Vert h^1\Vert_K\leq \Vert h^2\Vert_K$, but for our purposes it is convenient to omit this condition.
\end{Def}

%It was further shown in~\cite{HarSopr} that if a basis $(h^1,h^2)$ of $\Z^2$ is reduced with respect to a convex body $P\subset\R^2$ then matrix $A$ whose rows are $h^1$ and $h^2$ computes $\lss(P)$, and therefore the task of finding $\lss$ reduces to finding a basis which is reduced with respect to $P$.  
%

The algorithm for finding a reduced basis of $\R^2$ with respect to an arbitrary origin-symmetric convex body $K$  is explained and analyzed in~\cite{KaibSch}. A modified version of this algorithm is discussed in~\cite{HarSopr}. 

%\begin{Def} Let $\Vert h\Vert =\w_h(P)$, where $P$ is a convex body in $\R^3$. We say that a basis $(h^1, h^2, h^3)$ of the integer lattice $\Z^3\subset\R^3$ is {\it reduced} (with respect to $P$) if 
%\begin{itemize}
%\item $\Vert h^1\Vert\leq \Vert h^2\Vert\leq \Vert h^3\Vert$;
%\item $\Vert h^1\pm h^2\Vert \geq \Vert h^2\Vert$;
%\item  $\Vert m h^1+ n h^2+h^3\Vert \geq \Vert h^3\Vert$
%\end{itemize}
%\end{Def}

\section{Reduced Basis Computes the Lattice Size}
Let $P\subset\R^2$ be a convex body. 

\begin{Def}We define the naive lattice size  $\nls(P)$ to be the smallest $l$ such that $T(P)\subset l\Delta$, where $T\colon\R^2\to\R^2$ is a combination of a matrix multiplication by a 
diagonal matrix $A$ with entries $\pm 1$ on the main diagonal, and a lattice translation. Further, let
\begin{eqnarray*}
l_1(P)&:=&\max\limits_{(x,y)\in P}(x+y)-\min\limits_{(x,y)\in P}x-\min\limits_{(x,y)\in P}y,\\
l_2(P)&:=&\max\limits_{(x,y)\in P} x+\max\limits_{(x,y)\in P}y-\min\limits_{(x,y)\in P}(x+y),\\
l_3(P)&:=&\max\limits_{(x,y)\in P} y-\min\limits_{(x,y)\in P}x +\max\limits_{(x,y)\in P} (x-y),\\
l_4(P)&:=&\max\limits_{(x,y)\in P} x-\min\limits_{(x,y)\in P}y+\max\limits_{(x,y)\in P}(y-x).
\end{eqnarray*}
Then $\nls(P)$ is the smallest of these four values.
\end{Def}

\begin{Def} We say that basis $(h^1,h^2)$ of $\Z^2$ computes $\ls(P)$ if for matrix $A$ with rows $h^1, h^2$ we have $\ls(P)=\nls(AP)$.
\end{Def}

\begin{Ex}
Let $P=\conv\{(0,0), (4,1), (5,2)\}$. Then, as demonstrated in the diagram, $l_1(P)=7$, $l_2(P)=7$, $l_3(P)=5$, and $l_4(P)=5$, and hence $\nls(P)=5$.
\begin{figure}[h]
\includegraphics[scale=.5]{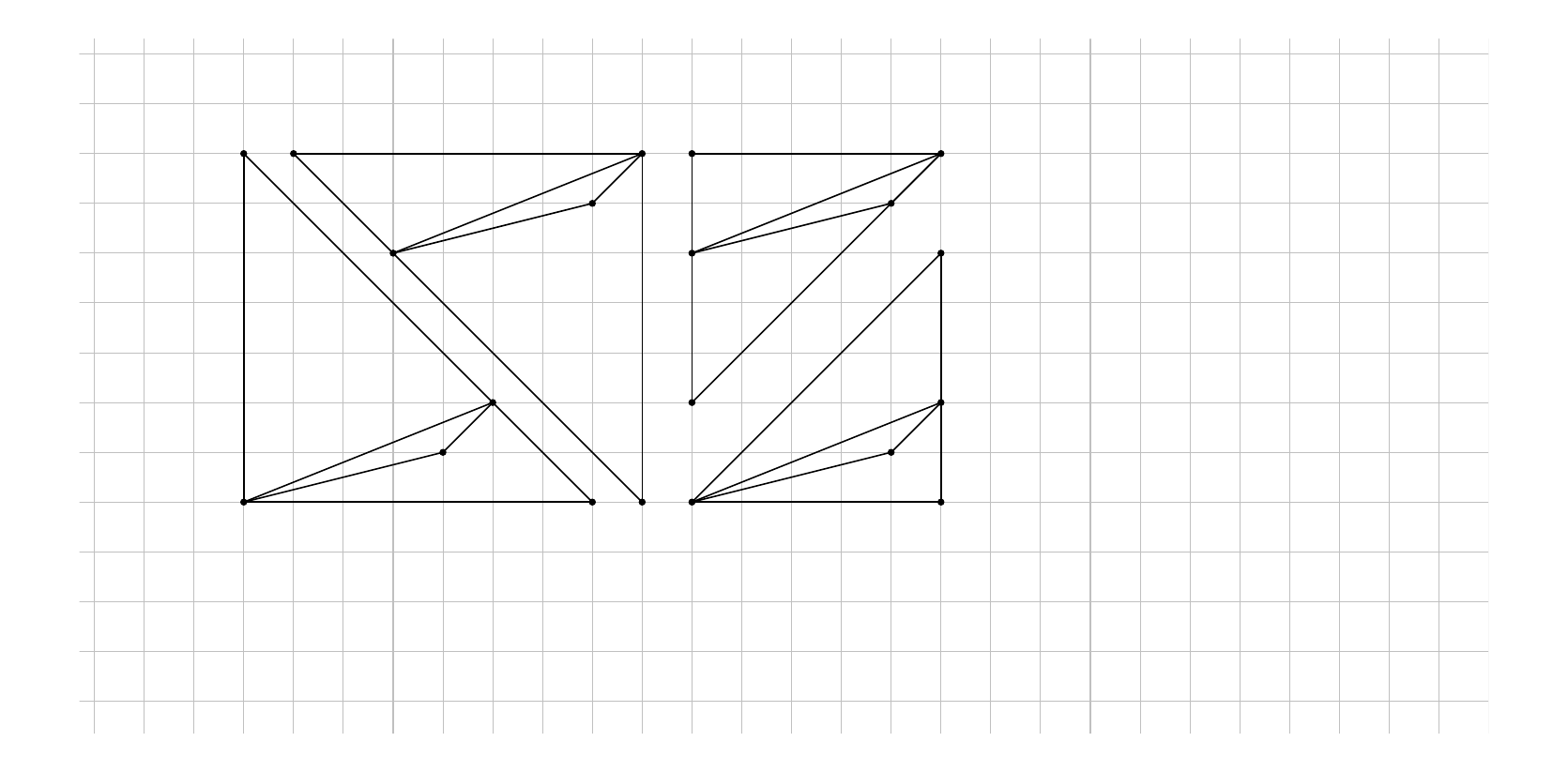}
%\caption{}
%\label{F:polygon}
\end{figure}
If we apply $A=\left[\begin{matrix} 1&-4\\0&1\end{matrix}\right]$ to $P$ we get a triangle with the vertices $(0,0), (0,1), (-3,2)$.
Then $l_1(AP)=4$, $l_2(AP)=3$, $l_3(AP)=5$, and $l_4(AP)=4$, so $\nls(AP)=3$.
Since $P$ has a lattice point inside, while $2\Delta$ does not, it is impossible to unimodularly map $P$ inside $2\Delta$.  We conclude that $\ls(P)=3$, and $A=\left[\begin{matrix}1&-4\\0&1\end{matrix}\right]$ 
and the lattice basis $\left((1,-4), (0,1)\right)$ compute the lattice size of $P$. 
\end{Ex}

We first record a few straight-forward observations.

\begin{Lemma}\label{L:applyA}
For  $h\in\R^2$ and $A\in{\rm GL}(2,\Z)$ we have  $\w_{h}(AP)=\w_{A^{T}h}(P)$. Therefore, if the rows of  $A$ are $h^1$ and $h^2$, then $\w_{e^1}(AP)=\w_{h^1}(P)$ and $\w_{e^2}(AP)=\w_{h^2}(P)$.
\end{Lemma}
\begin{proof}
For any $h\in\R^2$ we have
$$\w_{h}(AP)=\max\limits_{x\in P} h\cdot (Ax)- \min_{x\in P} h\cdot (Ax)=\max\limits_{x\in P} (A^{T}h)\cdot x- \min_{x\in P} (A^{T}h)\cdot x=\w_{A^{T}h}(P).
$$
\end{proof}

\begin{Lemma}\label{L:symmetries} 
\begin{itemize}
\item[(a)] If $P$ is reflected in the line $x=y$, then $l_1$ and $l_2$ do not change, while $l_3$ and $l_4$ switch with each other.
\item[(b)] If $P$ is reflected in the $x$-axis, then $l_1$ switches  with $l_3$, while $l_2$ switches  with $l_4$.
 \item[(c)] If $P$ is reflected in the $y$-axis, then $l_1$ switches  with $l_4$, while $l_2$ switches  with $l_3$.
\item[(d)] The naive lattice size $\nls(P)$ is preserved under the reflection in the line $y=x$ as well as under the reflections in the $x$- and $y$-axis.
\end{itemize}
\end{Lemma}

\begin{Lemma}\label{L:ls_properties} $\ $
\begin{itemize}
\item[(a)] 
\begin{eqnarray*}
l_1\left(\left[\begin{matrix}a&b\\c&d\end{matrix}\right]P\right)&=&\max\limits_{(x,y)\in P} ((a+c)x+(b+d) y)\\ & -& \min\limits_{(x,y)\in P} (ax+by) -  \min\limits_{(x,y)\in P} (cx+dy);\\
\end{eqnarray*}
\item[(b)]  $l_1\left(\left[\begin{matrix}a&b\\c&d\end{matrix}\right]P\right)=l_1\left(\left[\begin{matrix}c&d\\a&b\end{matrix}\right]P\right)$;
\item[(c)] \begin{eqnarray*} l_1\left(\left[\begin{matrix}a&b\\c&d\end{matrix}\right]P\right)&=&l_1\left(\left[\begin{matrix}-(a+c)&-(b+d)\\c&d\end{matrix}\right]P\right)\\
&=&l_1\left(\left[\begin{matrix} a&b\\-(a+c)&-(b+d)\end{matrix}\right]P\right).
\end{eqnarray*}
\end{itemize}
\begin{proof} 
First equality of part (c) is equivalent to claiming that $l_1(AP)=l_1(SAP)$, where $A=\left[\begin{matrix}a&b\\c&d\end{matrix}\right]$ and $S=\left[\begin{matrix}-1&-1\\0&1\end{matrix}\right]$, so replacing $AP$ with $P$, it is enough
to show that $l_1(SP)=l_1(P)$. We have
\begin{eqnarray*}
l_1(SP)&=&\max\limits_{(x,y)\in P}(-x)-\min\limits_{(x,y)\in P}(-x-y)-\min\limits_{(x,y)\in P} y\\
&=&\max\limits_{(x,y)\in P}(x+y)-\min\limits_{(x,y)\in P}x-\min\limits_{(x,y)\in P} y=l_1(P).
\end{eqnarray*}
\end{proof}
\end{Lemma}

\begin{Lemma}\label{L:lwidth}
Suppose that for $A=\left[\begin{matrix} a&b\\c&d\end{matrix}\right]$ we  have $l_1:=l_1(AP)< l$.
Then  the lattice width of $P$ in the three directions, $(a,b)$, $(c,d)$, and $(a+c,b+d)$, is less than $l$. 
\end{Lemma}
\begin{proof} 
We have $AP\subset l_1\Delta$, which implies that $\w_{e^1}(AP)\leq \w_{e^1} (l_1\Delta)=l_1$. 
Hence using Lemma~\ref{L:applyA} we get
$$\w_{(a,b)}(P)=\w_{e^1}(AP)\leq l_1<l.$$
The same argument applied to directions  $e^2$ and $e^1+e^2$ gives the remaining two conlcusions.
\end{proof}

\begin{Th}\label{T:main} Let $P\subset\R^2$ be a convex body and let $K=\left(P+(-P)\right)^{\circ}$. If the standard basis  $(e^1,e^2)$ is reduced with respect to $K$ then $\ls(P)=\nls(P)$.
\end{Th}

\begin{proof}
We will use the notation $\Vert h\Vert=\Vert h\Vert_K$ for $\w_h(P)$.  First, shift $P$ so that $\min\limits_{(x,y)\in P} x= \min\limits_{(x,y)\in P} y= 0$ and denote $$l:=l_1(P)=\max\limits_{(x,y)\in P}(x+y).$$
Next, denote  $k=\max\limits_{(x,y)\in P} x$; $m=\max \limits_{(x,y)\in P} y$; and 
$s=\min\limits_{(x,y)\in P}(x+y)$. Note that since there is a point of the form $(a,0)$ in $P$ we get $s\leq a\leq k$. Similarly, $s\leq m$.
We have $\Vert e^1\Vert= k$, $\Vert e^2\Vert= k$, and $\Vert e^1+e^2\Vert=l-s$. Since the standard basis is reduced, we get 
$l-s=\Vert e^1+e^2\Vert\geq \Vert e^1\Vert= k$ and  $l-s=\Vert e^1+e^2\Vert\geq \Vert e^2\Vert= m.$
This implies that $l\geq 2s$ since $k\geq s$ and $m\geq s$.
Hence 
\begin{equation}\label{e:2(1,1)}
2\Vert e^1+e^2\Vert=2l-2s=l+(l-2s)\geq l\geq \nls(P).
\end{equation}
Note that we also have $k+m\geq l$ since $P\subset [0,k]\times [0,m]\subset (k+m)\Delta$. We record the obtained relations for future use
\begin{equation}\label{e:all}
l\geq k+s,\ \ l\geq m+s, \ \ l\geq 2s,\ \ k+m\geq l.
\end{equation}

Pick a primitive direction $(a,b)\in\Z^2$.  Our goal is to show that $\Vert ae^1+be^2\Vert\geq \nls(P)$ for many such $(a,b)$. Note that if we have a reduced basis, we can flip its vectors and their signs and the obtained basis will also be reduced. 
Also, as we observed in Lemma~\ref{L:symmetries}, the naive lattice size $\nls(P)$ is invariant with respect to the reflections in the coordinate axes and the origin. Hence  without loss of generality we can assume that $a\geq b\geq 0$.
Using the triangle inequality we get 
$$\Vert ae^1+be^2\Vert+(a-b)\Vert e^2\Vert=\Vert ae^1+be^2\Vert+\Vert (a-b)e^2\Vert\geq a\Vert e^1+e^2\Vert.
$$
On the other hand, since $(e^1,e^2)$ is reduced we get
$$\Vert ae^1+be^2\Vert+(a-b)\Vert e^2\Vert\leq \Vert ae^1+be^2\Vert+(a-b)\Vert e^1+e^2\Vert.
$$
Hence for $b\geq 2$, combining the above two lines and (\ref{e:2(1,1)}), we get
$$\Vert ae^1+be^2\Vert \geq b \Vert e^1+e^2\Vert\geq 2\Vert e^1+e^2\Vert\geq \nls(P).
$$

We have checked that $\Vert ae^1+be^2\Vert(P)\geq\nls(P)$ for all primitive directions  $(a,b)$ with $\min\{|a|, |b|\}\geq 2$. 
Without loss of generality we can assume that $k\leq m$, that is, $\Vert e^1\Vert\leq \Vert e^2\Vert$. If  $b\geq a=1$
then
$$bm=b\Vert e^2\Vert\leq\Vert e^1+be^2\Vert +\Vert e^1\Vert=\Vert e^1+be^2\Vert+k,
$$
so for $b\geq 3$ we get
$$\Vert e^1+be^2\Vert\geq bm-k\geq (b-1)m\geq 2m\geq  k+m\geq l\geq  \nls(P).
$$

We have checked that under the assumption $k\leq m$ we have $\Vert ae^1+be^2\Vert\geq\nls(P)$ for all primitive $(a,b)$ except, possibly, for the ones in the set 
$$R=\{(a,\pm 1), (\pm 1,0), (\pm 1,\pm 2)\mid a\in\Z\}.
$$
It follows from  Lemma~\ref{L:lwidth} that it  remains to show that if the rows and the row sum of a unimodular matrix $A$ are in $R$ then $l_1(AP)\geq \nls(P)$.

Let first $A=\left[\begin{matrix}a&1\\1&0\end{matrix}\right]$ with $a\geq 0$. Then
$$l_1(AP)=\max\limits_{(x,y)\in P} ((a+1)x+y)-\min\limits_{(x,y)\in P} (ax+y)-\min \limits_{(x,y)\in P} (x).
$$

\vspace{.2cm}
 \begin{center} 
 \includegraphics[scale=.7]{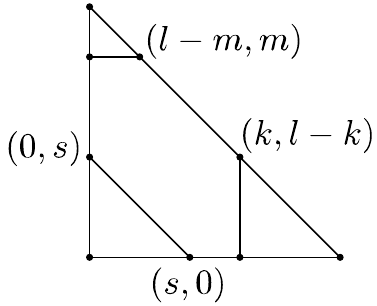}
\end{center}
Since $P\subset l\Delta$ has a point on the segment that joins points $(k,l-k)$ and $(l-m,m)$ we have
$$
\max\limits_{(x,y)\in P}((a+1)x+y)\geq \min\{(a+1)k+(l-k), (a+1)(l-m)+m\}=(a+1)(l-m)+m,
$$
where we used $k+m\geq l$ from~(\ref{e:all}). Using this together with $\min\limits_{(x,y)\in P} (ax+y)\leq as$ we get
$l_1(AP)\geq (a+1)(l-m)+m-as$. Then $(a+1)(l-m)+m-as\geq l$ is equivalent to $a(l-m)\geq as$, which holds true since $a\geq 0$ and $l\geq m+s$.

%Let next $A=\left[\begin{matrix}1&1\\0&1\end{matrix}\right]$. 
%We can assume that $l_1(P)\leq l_2(P)$ by reflecting $P$ in the origin, if necessary.
%Then we have $l_1(P)=l\leq k+m-s=l_2(P)$. This implies
%\begin{eqnarray*}
%l_1(AP)&=&\max\limits_{(x,y)\in P}(x+2y)-\min \limits_{(x,y)\in P} (x+y)-\min\limits_{(x,y)\in P} y\\
%&=&\max\limits_{(x,y)\in P}(x+2y)-s\geq 2m-s\geq k+m-s\geq l\geq \nls(P).
%\end{eqnarray*}

For $A=\left[\begin{matrix}1&1\\0&1\end{matrix}\right]$ we get
\begin{eqnarray*}
l_1(AP)&=&\max\limits_{(x,y)\in P}(x+2y)-\min \limits_{(x,y)\in P} (x+y)-\min\limits_{(x,y)\in P} y\\
&=&\max\limits_{(x,y)\in P}(x+2y)-s\geq 2m-s\geq k+m-s=l_2(P)\geq \nls(P).
\end{eqnarray*}

Suppose that we showed that for a matrix $A$ we have $l_1(AP)\geq \nls(P)$ whenever  $(e^1,e^2)$ is reduced with respect to $P$.
Let $Q$ be the reflection of $P$ in the $x$-axis. Then since $(e^1,e^2)$ is also reduced with respect to $Q$, we have also shown that $l_1(AQ)\geq \nls(Q)$. This implies that 
 $$l_1\left(\left[\begin{matrix}-a&b\\-c&d\end{matrix}\right] P\right)=l_1\left(\left[\begin{matrix}a&b\\c&d\end{matrix}\right] \left[\begin{matrix}-1&0\\0&1\end{matrix}\right] P\right)=l_1\left(A Q\right)\geq \nls(Q)=\nls(P).
 $$
Hence if we proved that $l_1(AP)\geq \nls(P)$ for matrix $A$, then  the same conclusion holds true for a matrix that can be obtained from $A$ by flipping one or both column signs.

Our next goal is to show that unimodular matrices whose rows and the row sum are in the set $R$ reduce to either $A=\left[\begin{matrix}1&1\\0&1\end{matrix}\right]$ or $\left[\begin{matrix}a&1\\1&0\end{matrix}\right]$ with $a\geq 0$ by flipping column signs and using reductions described  in part(c) of Lemma~\ref{L:ls_properties}.

Since  $(\pm 1,0)$ and  $(\pm 1,\pm 2)$ cannot be the two rows of a unimodular matrix, we can assume that first row of $A$ is of the form $(a,\pm 1)$. If now the second row is of the form $(\pm1,\pm 2)$, that is, $A=\left[\begin{matrix}a&1\\ \pm 1&-2\end{matrix}\right]$ or  
$A=\left[\begin{matrix}a& -1\\ \pm 1&2\end{matrix}\right]$ (the row sum also has to be in $R$), using Lemma~\ref{L:ls_properties},  we can 
change the second row to one of the form  $(b,\pm 1)$ for some $b\in\Z$ and get a matrix of the form $\left[\begin{matrix}a& \pm1\\ b&\pm 1\end{matrix}\right]$ for some $a,b\in\Z$.

If the second row of $A$ is of the form $(\pm 1,0)$ we get $A=\left[\begin{matrix}a& \pm1\\ \pm 1&0\end{matrix}\right]$ for some $a\in\Z$. Switching column signs we obtain a matrix of the form $\left[\begin{matrix}a& 1\\ 1&0\end{matrix}\right]$, where $a\in\Z$.

We now identify unimodular matrices both of whose rows are of the form $(a,\pm 1)$.  Let first $A=\left[\begin{matrix}a&1\\b&-1\end{matrix}\right]$ for some $a,b\in\Z$.
Then the sum of rows $(a+b,0)$ is  in $R$ we get $a+b=\pm 1$. Hence $A=\left[\begin{matrix}a&1\\\pm 1-a&-1\end{matrix}\right]$ for some $a\in\Z$ and using Lemma~\ref{L:ls_properties} we can replace $A$ with
 $\left[\begin{matrix}a&1\\\pm 1&0\end{matrix}\right]$, and switching the sign of the first column, if necessary, we get a matrix of the form $\left[\begin{matrix}a&1\\1&0\end{matrix}\right]$ with $a\in\Z$.

Next let $A=\left[\begin{matrix}a&1\\b&1\end{matrix}\right]$ for  $a,b\in\Z$. Then $a+b=\pm 1$ and $a-b=\pm 1$, where for  the second condition we used $\det A=\pm 1$.
Switching rows if necessary we get $A=\left[\begin{matrix}1&1\\0&1\end{matrix}\right]$ or $\left[\begin{matrix}-1&1\\0&1\end{matrix}\right]$. 
If $A$ is of the form $\left[\begin{matrix}a&-1\\b&-1\end{matrix}\right]$ we get $A=\left[\begin{matrix}1&-1\\0&-1\end{matrix}\right]$ or $\left[\begin{matrix}-1&-1\\0&-1\end{matrix}\right]$. All of these matrices reduce to $\left[\begin{matrix}1&1\\0&1\end{matrix}\right]$ after we flip column signs.

Next note that using Lemma~\ref{L:ls_properties} we can pass from $A=\left[\begin{matrix}-a& 1\\ 1&0\end{matrix}\right]$  to $\left[\begin{matrix}a-1& -1\\  1&0\end{matrix}\right]$ and switching the sign of the second column we pass to the matrix $\left[\begin{matrix}a-1& 1\\  1&0\end{matrix}\right]$. This shows that  matrices of the form $\left[\begin{matrix} a& 1\\ 1&0\end{matrix}\right]$ with $a\in\Z$ reduce to matrices of the same form with $a\geq 0$, for which we have checked that $l_1(AP)\geq \nls(P)$.

We have checked that for all unimodular $A$ we have $l_1(AP)\geq \nls(P)$, which implies $\ls(P)=\nls(P)$.
\end{proof}

\vspace{0cm}
\begin{Cor}~\label{C:cormain} Let $P\subset\R^2$ be a convex body. If the  basis  $(h^1,h^2)$ of $\Z^2$ is reduced with respect to $P$ then $\ls(P)=\nls(AP)$, where $A$ is a matrix whose rows are $h^1$ and $h^2$. 
\end{Cor}
\begin{proof}   
By Lemma~\ref{L:applyA} we have 
$$\w_{e^1}(AP)=\w_{h^1}(P),\  \w_{e^2}(AP)=\w_{h^2}(P),\  \ {\rm and}\  \ \w_{e^1\pm e^2}(AP)=\w_{h^1\pm h^2}(P),
$$ 
and it follows that the standard basis $(e^1,e^2)$ is reduced with respect to $AK$. 
By Theorem~\ref{T:main} this implies that $\ls(P)=\ls(AP)=\nls(AP)$.
\end{proof}

The algorithm for finding a reduced basis is explained and analyzed in~\cite{HarSopr} and~\cite{KaibSch}. In each iteration of the algorithm, starting with basis $(h^1,h^2)$ with $\Vert h^1\Vert\leq \Vert h^2\Vert$, one finds $m$ that minimizes $\Vert mh^1+h^2\Vert$. If $\Vert mh^1+h^2\Vert\geq \Vert h^1\Vert$ then the basis $(h^1, mh^1+h^2)$ is reduced. Otherwise, the next iteration is applied to the basis $(mh^1+h^2,h^1)$.
In the example below we find a reduced basis and compute the lattice size.

\begin{Ex}
Let $P=\conv\{(0,0), (3,5), (7,9), (8,12)\}$ and let $h^1=e^1, h^2=e^2$. Then $\Vert e^1\Vert_P=8$ and $\Vert e^2\Vert_P=12$. For $m\in\Z$  we have
$$\Vert me^1+e^2\Vert _P= \max\{0,3m+5, 7m+9, 8m+12\}- \min\{0,3m+5, 7m+9, 8m+12\}.
$$
Hence $\Vert me^1+e^2\Vert_P\geq 8m+12\geq 12$ for $m\geq 0$ and $\Vert me^1+e^2\Vert_P\geq -7m-9\geq 5$ for $m\leq -2$.
Also,  $\Vert -e^1+e^2\Vert_P=4<\Vert e^1\Vert_P=5$ and hence we should pass to the polygon $AP$, where $A=\left[\begin{matrix}-1& 1\\ 1&0\end{matrix}\right]$. We get 
$AP=\conv\{(0,0), (2,3), (2,7), (4,8)\}$. Now the minimum of $\Vert me^1+e^2\Vert_{AP}$ is achieved at $m=-2$  and $\Vert -2e^1+e^2\Vert_{AP}=4$.
Since $\Vert -2e^1+e^2\Vert_{AP}\geq \Vert e^1\Vert_{AP}$ the algorithm terminates at this step. Hence matrix \linebreak $B=\left[\begin{matrix}-1& 1\\ 3&-2\end{matrix}\right]=\left[\begin{matrix}1& 0\\ -2&1\end{matrix}\right]\left[\begin{matrix}-1& 1\\ 1&0\end{matrix}\right]$
computes the lattice size and basis $((-1,1), (3,-2))$ is reduced with respect to $P$. We get $BP=\conv\{(0,0), (2,-1), (2,3), (4,0)\}$. We conclude that $\ls(P)=\nls(BP)=\min\{8,6,5,6\}=5$.
\end{Ex}

\section{Counterexample in the 3-space.}

For a convex body $P\subset\R^3$ define { $\nls(P)$ to be the smallest $l$ such that $P$ is contained in  $l\Delta$  after a transformation which is  a composition of $A=\left[\begin{matrix}\pm 1&0&0\\0&\pm 1&0\\0&0&\pm 1\end{matrix}\right]$ and a lattice translation. We also define
$$l_1(P):=\max\limits_{(x,y,z)\in P}(x+y+z)-\min\limits_{(x,y,z)\in P}x-\min\limits_{(x,y,z)\in P}y-\min\limits_{(x,y,z)\in P}z.
$$

Let $P\subset\R^3$ be a convex body and let $K=\left(P+(-P)\right)^{\circ}$. 

\begin{Def}
A basis $(h^1,h^2,h^3)$ of $Z^3\subset\R^3$ is reduced with respect to $K$ if 
\begin{itemize}
\item[(1)] $\Vert h^1\Vert\leq \Vert h^2\Vert\leq \Vert h^3\Vert$,
\item[(2)] $\Vert h^1\pm h^2\Vert\geq \Vert h^2\Vert$, and
\item[(3)] $\Vert mh^1+ nh^2+h^3\Vert\geq \Vert h^3\Vert$ for all $m,n\in\Z$.
\end{itemize}
\end{Def}

Note that condition (2) is equivalent to a seemingly stronger requirement that $\Vert mh^1+ h^2\Vert\geq \Vert h^2\Vert$ for all $m\in\Z$.
Assume that $\Vert h^1\pm h^2\Vert\geq \Vert h^2\Vert$. Then for $m\in\Z$ we have
$$\Vert mh^1+h^2\Vert+(|m|-1)\Vert h^2\Vert\geq |m|\Vert h^1+\operatorname{sgn}(m)h^2\Vert\geq |m|\Vert h^2\Vert,
$$
which implies that $\Vert mh^1+ h^2\Vert\geq \Vert h^2\Vert$. Hence the above definition is a natural generalization of Definition~\ref{D:def_red2}

It is shown in \cite{HarSopr} that if a basis $(h^1,h^2,h^3)$ is reduced then matrix $A$ whose rows are $h^1, h^2,$ and $h^3$ computes $\lss(P)$, which led to a fast algorithm for computing $\lss(P)$ in dimension 3.
Hence it is natural to ask  whether for such $A$ we  have $\ls(P)=\nls(AP)$. We can also relax this question and ask whether there exists a reduced basis  $(h^1,h^2,h^3)$ such that for the corresponding matrix $A$ we have
$\ls(P)=\nls(AP)$. The answer to both of these questions is negative, as  the example below shows.

\begin{Ex}\label{E:counter_example} Let $P=\conv \{(1,1,2), (4,4,4), (0,2,2), (3,0,3), (4,3,0)\}\subset \R^3$. 
Then $\w_{e^1}(P)=\w_{e^2}(P)=\w_{e^3}(P)=4$. Let's check that for primitive $(a,b,c)\in\Z^3$ we have $\w_{(a,b,c)}(P)>4$ unless one of $a,b$, and $c$ is $\pm 1$ and the other two numbers are 0.
We have
\begin{align*}
\w_{(a,b,c)}(P)&=\max\{a+b+2c,4a+4b+4c,2b+2c, 3a+3c, 4a+3b\}\\
&-\min\{a+b+2c,4a+4b+4c,2b+2c, 3a+3c, 4a+3b\}.
\end{align*}
We can assume that $c\geq 0$. Let first $c\geq 1$. Then if $b\geq 1$ we have $$\w_{(a,b,c)}(P)\geq (4a+4b+4c)-(4a+3b)=b+4c\geq 5.$$ For $b\leq -1$ and $a\geq 1$, or $b= 0$ and $a\geq 2$ we have
$$\w_{(a,b,c)}(P)\geq (3a+3c)-(a+b+2c)=2a-b+c\geq 5.$$ 
 For $b\leq -1$ and $a\leq 0$, or $b= 0$ and $a\geq -2$ we get
 $$\w_{(a,b,c)}(P)\geq (3a+3c)-(4a+3b)=-a-3b+3c\geq 5.$$ 
If $a=1$ and  $b=0$ then $\w_{(a,b,c)}(P)\geq (4a+4b+4c)-(2b+2c)=4+2c\geq 6$. If $a=-1$ and  $b=0$ then $\w_{(a,b,c)}(P)\geq (a+b+2c)-(4a+3b)=2+2c\geq 5.$

Next, let $c=0$, so we can assume that $a\geq 0$.
If $a\geq 1$ and $b\geq 1$ we have $\w_{(a,b,c)}(P)\geq (4a+4b)-(2b)\geq 6.$ If $a\geq 1$ and $b\leq -1$ we have $\w_{(a,b,c)}(P)\geq 3a-2b\geq 5.$ 
If $a=0$ and $b\geq 1$ then $\w_{(a,b,c)}(P)\geq 4b\geq 8$ unless $b=1$.

We have checked that the only primitive directions in which the lattice width of $P$ is at most 4 are $\pm e^1$, $\pm e^2$, and $\pm e^3$ and the lattice width of $P$ in these directions is 4.
Hence the bases that one can obtain from $(\pm e^1,\pm e^2, \pm e^3)$ by switching the order of the vectors and flipping their signs are reduced. We next show that there are no other reduced bases here. 

Let now $(h^1,h^2,h^3)$ be reduced. It is shown in Theorem~3.3 in~\cite{HarSopr} that   $\Vert ah^1+bh^2+ch^3\Vert\geq\Vert h^3\Vert$ for  all $a,b,c\in\Z$ with nonzero $c$, except, possibly, for one primitive direction $\pm(ah^1+bh^2+ch^3)$ with 
$|a|=|b|=1$ and $|c|=2$. Hence if $\Vert h^3 \Vert \geq 5$ we can conclude that two standard vectors, say, $e^1$ and $e^2$, belong to the $(h^1,h^2)$-plane. 
Since, as was shown in Theorem~2.3 in~\cite{HarSopr}, for $a,b\in\Z$ we have $\Vert a h^1+bh^2\Vert\geq \Vert h^2\Vert$
 for all nonzero $b$, we can conclude that bases $(h^1, h^2)$ and $(e^1,e^2)$ coincide up to changing the order of vectors and flipping their signs. Since $\Vert e^3\Vert=4$, this would imply that $e^3=ae^1+be^2+ch^3$, where $|a|=|b|=1$ and $|c|=2$, but then 
the change of basis matrix from $(h^1,h^2,h^3)$ to the standard basis  would have determinant $2$ or $-2$, which is impossible. We conclude that $\Vert h^3\Vert =4$ and $(h^1,h^2, h^3)$ is obtained from the standard basis by switching the order of the vectors and flipping their signs.

Next, one can easily check that $\nls(P)=8$. Using the obtained description of reduced bases we conclude that for matrix $A$ whose rows form a reduced basis we also have $\nls(AP)=8$.
For  $B=\left[\begin{matrix}-1&0&0\\0&-1&0\\ 0&1&1\end{matrix}\right]$ we get $$BP=\conv\{(-1,-1,3), (-4,-4,8), (0,-2,4), (-3,0,3), (-4,-3,3)\}.$$
Hence $l_1(BP)=2-(-4)-(-4)-3=7$ and we conclude that $\ls(P)\leq 7$, so there is no reduced basis here such that for the corresponding matrix $A$ we have  $\ls(P)=\nls(AP)$. 
\end{Ex}

\subsection*{Acknowledgments} 
Work of Soprunova and Tierney was partially supported by NSF Grant DMS-1156798.

%\bibliographystyle{amsplain}
%\bibliography{bibl}

\begin{thebibliography}{6}

\bibitem{Arnold} V. Arnold, \emph{Statistics of integral convex polygons}, Functional Analysis and Its
Applications 14(2), 1-3 (1980)


\bibitem{Barvinok} A. Barvinok, \emph{Integer Points in Polyhedra}, Zurich Lectures in Advanced Mathematics, ISBN: 9783037190524 (2008)


\bibitem{BarPach} I. B\'ar\'any and J. Pach, \emph{On the number of convex lattice polygons},  Combinatorics,
Probability and Computing {\bf 1}, Issue 4 (1992)


\bibitem{BrownKasp} G. Brown, A.  Kasprzyk, \emph{Small polygons and toric codes}, Journal of Symbolic
Computation, {\bf 51}  p. 55 (2013)

\bibitem{CasCools} W. Castryck , F. Cools, \emph{The lattice size of a lattice polygon}, Journal of Combinatorial Theory Series A {\bf 136}, Issue C, 64-95 (2015)


\bibitem{HarSopr} A. Harrison, J. Soprunova, \emph{Lattice Size of 2D and 3D polytopes with respect to the unit cube}, preprint, arXiv:1709.03451 (2020)


\bibitem{KaibSch} M. Kaib, C. Schnorr, \emph{The Generalized Gauss Reduction Algorithm}, Journal of Algorithms
21(3) 565-578 (1996)

\bibitem{LagZieg} J. Lagarias, G. Ziegler, \emph{Bounds for lattice polytopes containing a fixed number of interior points in a sublattice}, Canadian Journal of Mathematics 43(5),
1022-1035 (1991)

\bibitem{LovScarf} L. Lov\'asz, H. Scarf \emph{The Generalized Basis Reduction Algorithm}, Mathematics of Operations Research {\bf 17}, Issue 3, 751-764 (1992)



\bibitem{Schicho} J. Schicho, \emph{Simplification of surface parametrizations -- a lattice polygon approach}, Journal of Symbolic Computation {\bf 36}(3-4), 535-554 (2003).


\end{thebibliography}

\end{document}